\documentclass[12pt,reqno]{amsart}

\usepackage{amsmath,amsfonts,amssymb,amsthm,amsaddr}

\usepackage{color}
\usepackage{hyperref}

\usepackage{geometry}
\newtheorem{theorem}{Theorem}[section]
\newtheorem{lemma}{Lemma}[section]

\renewcommand{\L}{\mathsf{L}}

\begin{document}

\title[Eigenvalue bound for Schr\"{o}dinger operators]{Eigenvalue bound for Schr\"{o}dinger operators with unbounded magnetic field}

\author{Diana Barseghyan}
\address{Department of Mathematics, University of Ostrava, 30.~dubna 22, 70103 Ostrava, Czech Republic \& Department of Theoretical Physics, Nuclear Physics Institute, 25068 \v{R}e\v{z} near Prague, Czech Republic}
\email{diana.barseghyan@osu.cz}

\author{Baruch Schneider}
\address{Department of Mathematics, University of Ostrava, 30.~dubna 22, 70103 Ostrava, Czech Republic}
\email{baruch.schneider@osu.cz} 
 
\keywords{Eigenvalue bounds, radial
magnetic field, Lieb-Thirring inequalities, discrete spectrum, eigenvalue counting function}

\begin{abstract}
In this paper we consider  magnetic Schr\"{o}dinger operators  on the two-dimensional unit disk with a radially symmetric magnetic field which explodes to infinity at the boundary. We prove a bound for the eigenvalue moments and a bound for the number of negative eigenvalues for such operators.
\end{abstract}

\maketitle

\section{Introduction\label{sec1}} 
 
\subsection{}
The current paper deals with eigenvalue bounds for \emph{magnetic} Schr\"{o}dinger operators. However at first we recall some classical results for non-magnetic case.
Let $V(x)$ be a bounded measurable real-valued function on an open  set 
$\Omega\subset \mathbb{R}^d$, $d\ge 1$. We consider the Schr\"{o}dinger operator
\begin{equation}\label{non-magn.}
H_\Omega(0,V)=-\Delta_D^\Omega -V
\end{equation}
acting in $\L^2(\Omega)$ subject to the Dirichlet conditions on the boundary of $\Omega$; $0$ is the notation $H_\Omega(0,V)$ reflects the fact that there are no magnetic potential.  
Denote by $\{\lambda_j(\Omega, 0,V)\}_{j=1}^{N}$ the eigenvalues of $H_\Omega(0,V)$ located below the bottom of the essential spectrum of $H_\Omega(0,V)$. As usual we renumber the eigenvalues in the non-decreasing order and  repeat them according to their multiplicity. 
If $\Omega$ is bounded, the spectrum of $H_\Omega(0,V)$ is purely discrete, $N=\infty$, and the eigenvalues $\lambda_j(\Omega,0, V)$ accumulates at infinity. 
The main object of our studies
are  the so-called \emph{Riesz means} given by
\begin{gather}
\label{tr:def}
\mathrm{tr}(H_\Omega(0,V))_-^\sigma=\sum_{\lambda_j(\Omega,0, V)\le0} |\lambda_j(\Omega, V)|^\sigma,\quad\sigma\ge0.
\end{gather}
Here and in what follows the notation $f_\pm=(|f|\pm f )/2$ stays for the positive and negative parts of a number, a function or an operator. Note that for $\sigma=0$  the quantity in \eqref{tr:def} is  the number of non-positive eigenvalues of $H_\Omega(0,V)$.

The first classical result in this area concerns the behaviour of $\mathrm{tr}(H_\Omega(0,V))_-^\sigma$
in the strong coupling limit. Namely,
introducing a scaling parameter $\lambda>0$ and replacing the potential $V$ by $\lambda V$ one gets the  asymptotic formula   
\begin{equation}\label{Weyl}
\lim_{\lambda\to\infty}\lambda^{-\sigma-d/2}
\mathrm{tr}(H_\Omega(0,\lambda V))_-^\sigma = L_{\sigma, d}^{\mathrm{cl}}  \int_\Omega V_+(z)^{\sigma+d/2}dz,\ \sigma\ge 0
\end{equation}
with the semiclassical constant 
\begin{equation}\label{const}
L_{\sigma,d}^{\mathrm{cl}}=\frac{\Gamma(\sigma+1)}{(4\pi)^{
\frac{d}{2}}\Gamma
(\sigma+1+d/2)}.
\end{equation}
We assumed above that $V\in \L^{\sigma+d/2}(\Omega)$.
For  $\sigma=0,\, V\equiv {\rm const}$ this result goes back to H.~Weyl \cite{W12}, therefore \eqref{Weyl} is usually refered to as \emph{Weyl's law}.

The second classical result -- \emph{Lieb-Thirring inequality} -- was established by E.H.~Lieb and W.~Thirring in 
\cite{LT76}.
It states that the right-hand-side in \eqref{Weyl}  is not only the limit of the left-hand-side, but also an upper bound (up to a multiplicative constant).
Namely, for $\sigma >\mathrm{max}\{0,1-d/2\}$ and $V\in \L^{\sigma+d/2}(\Omega)$   the estimate
\begin{equation}
\label{Lieb:Thirring}
\mathrm{tr}(H_{\Omega}(0,V))_-^\sigma\le L_{\sigma,d} \int_{\Omega} V_+(z)^{\sigma+d/2}dz
\end{equation}
holds with certain positive constant $L_{\sigma,d}$. 
In fact, the above result was established in
\cite{LT76} for $\Omega=\mathbb{R}^d$, and then for an arbitrary domain $\Omega$
it holds immediately due to the inequality  
\begin{gather}
\label{tr:tr}
\mathrm{tr}(H_{\Omega}(0,V))_-^\sigma\le
\mathrm{tr}(H_{\mathbb{R}^d}(0,\widehat V))_-^\sigma,
\end{gather}
where $\widehat V$ is the extension of $V$ by zero to $\mathbb{R}^d\setminus\Omega$;  \eqref{tr:tr} follows easily from the min-max principle (see, e.g., \cite{RS78}). 

Note that estimate \eqref{Lieb:Thirring} remains valid  for
$\sigma=0,\,d\ge 3$. This result 
was established independently by  M.~Cwikel \cite{Cw77}, E.H.~Lieb \cite{Li80}, and G.V.~Rozenblyum \cite{Ro72,Ro76}. T.~Weidl \cite{Wei96} proved that \eqref{Lieb:Thirring} also holds for $d=1,\,\sigma={1/2}$.
However for $d=2$ and $\sigma=0$ \eqref{Lieb:Thirring} does not hold.
In this case one has the following estimate  established by K.~Chadan, N.N.~Khuri, A.~Martin, and T. T.~Wu in \cite{CKMW03} under the assumption that the potential \emph{$V$ is  radially symmetric}:
\begin{equation}\label{Chadan}
\mathrm{tr}(H_{\Omega}(0,V))_-^0\le 1 +\int_{\mathbb{R}^2} V_+(z) (1 +|\ln |z||)\,d z\,.
\end{equation}

\subsection{}
Despite the rigorous study of Schr\"{o}dinger operators \eqref{non-magn.}, there has been much less investigation of Schr\"{o}dinger operators with magnetic fields, which are in focus of the present paper. Let $\Omega$ be a open set in $\mathbb{R}^2$; in what follows the points in $\Omega$ will be denoted by $z$, its Cartesian coordinates will be denoted by $(x,y)$. Let
$$A=(A_1,A_2):\Omega\to \mathbb{R}^2\text{ (magnetic potential) 
},\quad V:\Omega\to \mathbb{R}\text{ (electric potential)}.$$
As above $V$ is assumed to be bounded and measurable. 
The two-dimensional 
magnetic Schr\"odinger operator is (formally) defined by
\begin{equation}
\label{operator}
H_\Omega(A, V) = (i\nabla+A )^2-V.
\end{equation} 
On $\partial\Omega$ we again prescribe the Dirichlet boundary conditions. 
The magnetic field $B$ is given by $$B=\mathrm{rot}\, A={\partial A_2\over \partial x}-{\partial A_1\over \partial y}.$$  
Again we denote the eigenvalues $H_\Omega(A, V)$ lying below the bottom of the essential spectrum by $\{\lambda_j(\Omega, A, V)\}_{j=1}^N$  renumbering them in the increasing order and with account of their multiplicities. 
Note, that if $\Omega$ is bounded and the vector potential $A$ satisfies mild regularity conditions, the magnetic Sobolev norm 
$$\|(i \nabla+A)u)\|_{\L^2(\Omega)}^2,\,\,u\in \mathcal{H}_0^1(\Omega)$$
is equivalent to the non-magnetic one, whence one can easily deduce  the discreteness of the spectrum of  $H_\Omega(A, V)$, i.e.  in this case one has $N=\infty$, and eigenvalues $\lambda_j(\Omega, A, V)$ accumulates to infinity. 

A.~Laptev and T.~Weidl \cite{LW00} proved that 
\begin{gather}
\label{LW}
\mathrm{tr}(H_{\mathbb{R}^d}(A, V))_-^\sigma\le L_{\sigma, d}^{\mathrm{cl}} \int_{\mathbb{R}^d} V_+(z)^{\sigma+d/2} dz,\quad \sigma\ge 3/2
\end{gather}
provided 
$A\in \L^2_{\mathrm{loc}}(\mathbb{R}^d)$ and $V\in \L^{\sigma+d/2}(\mathbb{R}^d)$.
By the minimax principle  estimate \eqref{LW} also
holds with an arbitrary bounded domain $\Omega $ instead of $\mathbb{R}^d$,
provided $A\in \L^2(\Omega)$ and $V\in \L^{\sigma+d/2}(\Omega)$.

\subsection{}
One of the models attracting considerable attention in the last ten years
concerns  magnetic Schr\"odinger operators on bounded domains $\Omega$
with magnetic fields satisfying
\begin{gather}
\label{grow}
B(z)\to\infty\text{ as }z\to\partial\Omega.
\end{gather}
Apparently, for the first time such model was treated in
\cite{CdVT10}, where the authors established the essential 
self-adjointness of $H_{\Omega}(A, V)$ (defined on $C^\infty_0(\Omega)$) 
under certain assumptions on the growth of $B$ near the boundary of $\Omega$.
The obtained results are of some technical interest due to their connection with to a special kind of magnetic confinement devices -- \emph{tokamacs}. 

In the current paper for the above model
we derive Lieb-Thirring-type inequality (Theorem~\ref{th1}) and  also
the estimate for the number of negative eigenvalues (Theorem~\ref{th2}) under the following 
restrictions: $\Omega$ is a unit disc, the magnetic field is radially symmetric with respect to the center of this disc, and also some additional conditions of the growth of $B$ near the boundary of $\Omega$ take place (cf.~\eqref{assumption2}).
It is important that magnetic potentials we deal with are not necessary
in $\L^2(\Omega)$ and therefore we are not able to apply the results of \cite{LW00}.

We formulate the main results in Section~\ref{sec2}.
Note, that some other eigenvalue bounds for operators \eqref{operator} with magnetic fields satisfying \eqref{grow} were also derived in \cite{BEKW16,BT19,T12} 
under more restrictive assumptions on $A$ and $B$.
At the end of Section~\ref{sec2} we compare the estimates obtained in these works  with  the estimates presented in the present paper. 
Their proof are given in Section~\ref{sec3}. 

\section{Results\label{sec2}}

Let $\Omega$ be the two-dimensional unit disk centered at the origin.
We denote the points in $\Omega$ by $z=(r,\theta)$, where 
$(r,\theta)$ are polar coordinates (with respect to the center of $\Omega$).  

We are given the bounded measurable function $V:\Omega\to\mathbb{R}$ (electric potential)
and the radially symmetric function $B:\Omega\to \mathbb{R}$ (magnetic field) satisfying
\begin{gather}
\label{assumption1}
\inf_{z\in\Omega}\,B(z)>0,\\
\label{assumption2}
B(z)=\frac{M}{(1-|z|)^\alpha}+g(|z|) 
\end{gather}
with some $\alpha\in (0,2]$, $M>0$, and a bounded measurable 
function $g:[0,\infty)\to\mathbb{R}$.
It is easy to see that,
up to a gauge transformation, the corresponding magnetic potential $A=(A_1,A_2)$
is given by
$$A_1(z)=-\sin\theta\cdot \int_0^r s B(s) ds ,\quad A_2=\cos\theta\cdot\int_0^r s B(s) ds,\quad z=(r,\theta).$$
Note that $A$ does not belong to $\L^2(\Omega)$ as $\alpha\ge3/2$.

We define the operator $H_\Omega (A, V)$ acting in $\L^2(\Omega)$ by differential operation 
\eqref{operator}, first on the smooth and compactly supported in $\Omega$ functions.
In view of \eqref{assumption1} and the well-known lower bound (see, e.g., \cite{AHS78}) 
$$(H_\Omega (A, V) (u), u)_{\L^2(\Omega)}\ge \int_\Omega \left(B(z)-\|V\|_{\L^\infty(\Omega)}\right) |u(z)|^2\,dz$$ one can construct the Friedrichs extension of $H_\Omega (A, V)$. 
For simplicity, we will use for this Friedrichs extension the same symbol  $H_\Omega(A, V)$. 

One can show that $H_\Omega(A, V)$ has a purely discrete spectrum. In fact, this fact will be established within the proof of Theorem~\ref{th1}.  We denote the increasingly ordered sequence of  the eigenvalues of $H_\Omega(A, V)$  by $\lambda_k =\lambda_k(\Omega, A, V)$, $k=1,2,3\dots$\,.\smallskip

Further we will also need the function $\widetilde V:\mathbb{R}_+\to\mathbb{R}$ given by
\begin{equation}\label{tildeV}
\widetilde{V}(r):=\underset{\theta\in[0,2\pi)}{\mathrm{ess\,sup}}V_+(r,\theta).
\end{equation}
Now we are in position to give the main results of this work.

\begin{theorem} 
\label{th1} 
For any $\sigma> 0$ the inequality
\begin{equation}
\label{main:est:1}
\mathrm{tr}\left(H_\Omega(A, V)\right)_-^\sigma\le C \int_0^1 \widetilde{V}^{\sigma+1}(r) r\,dr
\end{equation}
holds
with some positive constant $C=C(B,\sigma)$ depending on $B$ and on $\sigma$.
\end{theorem}

For the radially symmetric potential $V$ our estimate \eqref{main:est:1} coincides with the standard Lieb-Thirring inequality up to a constant depending on the magnetic field.

\begin{theorem}
\label{th2}
The estimate
\begin{gather}
\label{main:est:2}
\mathrm{tr}\left(H_\Omega(A, V)\right)_-^0\le 1+C_1(B)\int_0^1 \widetilde{V}_+(r)(1+|\ln r|)r \,dr
\end{gather}
holds
with some positive constant $C_1=C_1(B )$ depending on $B$.
\end{theorem}

For the radially symmetric potential $V$ our estimate \eqref{main:est:2} coincides with \eqref{Chadan} up to a constant depending on the magnetic field.

\subsection*{Discussion}
Estimates for Riesz means $\mathrm{tr}\left(H_\Omega(A, V)\right)_-^\sigma$  as $\sigma>0$ and 
 the magnetic field satisfies \eqref{grow} 
have been obtained in \cite{BEKW16,BT19}
under stronger restriction $B$ and $A$ comparing with those we treat in the present paper.
Namely, in \cite{BEKW16} the authors assumed that the total magnetic flux 
$\int_\Omega B(z)\,dz$ is less than $\pi$ (for $B$ satisfying \eqref{assumption2} this does not hold for $\alpha\ge 1$); 
in \cite{BT19} the right-hand side of the estimate explodes to infinity if $A\not\in \L^2(\Omega)$. 

As regard to  the number of negative eigenvalues,
we refer at first to the paper \cite{K11}, where the author treated  
 magnetic Schr\"odinger operators on $\mathbb{R}^2$
and for a large class of magnetic potential obtained 
the estimates resembling \eqref{Chadan} (as in our Theorem~\ref{th2}, with constants depending on the magnetic field),
again under the assumption that
magnetic potentials are in $\L^2_{\rm loc}(\mathbb{R}^d)$ -- thus we cannot use these results for all magnetic fields we treat in the present paper 
(recall that in the present paper the magnetic potential does not belong  to $\L^2(\Omega)$ if in the assumption \eqref{assumption2} $\alpha$ is larger or equal then $3/2$). 

If $\Omega$ is a disc and the magnetic field is radially symmetric and satisfies \eqref{grow}, some estimates (rather different from \eqref{main:est:2}) for  the number of negative eigenvalues
were obtained in \cite{T12} under additional assumption
$$B(z)\leq \frac{M}{(1-|z|)^\alpha},\ \alpha\in (0,3/2),\ M>0.$$
In this case the underlying magnetic potential is again square integrable on $\Omega$.

\section{Proof of Theorems~\ref{th1}-\ref{th2}\label{sec3}}

In what follows, we use the same notation $\widetilde V$ 
for the function of $r\in [0,\infty)$ defined by \eqref{tildeV} as well 
for the radially symmetric function of $z=(r,\theta)\in\Omega$, whose values
at $z$ with $|z|=r$ are defined  by \eqref{tildeV}. That is, $\widetilde V(r)=\widetilde V(z)$ as $z=(r,\theta)\in\Omega$.
\smallskip

At first we observe that
$
H_\Omega (A, V) \leq H_\Omega (A, \widetilde{V})
$. Then  by the minimax principle 
$$\mathrm{tr}\left(H_\Omega(A, V)\right)_-^\sigma\le\mathrm{tr}(H_\Omega(A, \widetilde{V}))_-^\sigma,\quad \sigma\ge0.
$$
Hence  it is sufficient to prove estimate (\ref{main:est:1}) with $H_\Omega(A, \widetilde{V})$ being replaced by $H_\Omega(A, V)$ in its left-hand-side.
\smallskip

Recall, that our magnetic potential $A(r,\theta)$ is given by
$$
A(r,\theta) = 
\left(-\frac{\Phi(r)}{r} \sin\theta, \frac{\Phi(r)}{r} \cos\theta\right),\ \text{where }
\Phi(r) := \int_0^r s B(s)\,ds.
$$

We denote by $h_m(B, \widetilde{V})$  the Friedrichs extensions of the operator being associated with the quadratic form $Q(h_m(B, \widetilde{V}))$ in $\L^2((0, 1), 2\pi r dr)$,
$$Q(h_m(B, \widetilde{V}))[v]= 2\pi \int_0^1  \left(|v'(r)|^2+\frac{(m-\Phi(r))^2}{r^2} |v(r)|^2-\widetilde{V}(r) |v(r)|^2\right)\,r\,dr,\ v\in C_0^\infty(0,1).$$ 
The action of this operator is given by
$$
h_m(B, \widetilde{V})=-\frac{\mathrm{d}^2}{\mathrm{d}r^2}-\frac{1}{r}\frac{\mathrm{d}}{\mathrm{d} r}+\frac{(m-\Phi(r))^2}{r^2}-\widetilde{V}(r).$$
The radial symmetry of our potentials implies (see~\cite{E96}) 
the direct sum decomposition
$$H_\Omega(A, \widetilde{V})=\bigoplus_{m=-\infty}^\infty h_m(B, \widetilde{V})$$ 
with respect to
the space decomposition
\begin{gather*}
\L^2(\Omega,  d x )=
\bigoplus_{m=-\infty}^\infty \L^2((0, 1), 2\pi r dr),\\ 
f \rightarrow (...,f_{-1},f_0,f_1,....)   \text{\quad with\quad }  f(r,\theta ) = \sum_{m=-\infty}^\infty e^{im\theta} f_m (r).
\end{gather*} 

Our strategy is to prove that  
\begin{gather*}
h_m(B, \widetilde{V})\ge \gamma h_m(0, \widetilde{V}/\gamma),
\end{gather*}
with some constant $\gamma>0$, and then to employ the standard 
Lieb-Thirring bound for the \emph{non-magnetic} Schr\"odinger operator
\begin{gather}
\label{nonmagnetic:dec}
H_\Omega(0,\widetilde V/\gamma)=\bigoplus_{m=-\infty}^\infty h_m(0, \widetilde{V}/\gamma).
\end{gather}

Due to \eqref{assumption1}
 $\Phi(r)\ge 0$, whence
we immediately conclude that
\begin{align}
\label{negative} 
h_m(B, \widetilde{V})\ge \gamma_0\left(-\frac{\mathrm{d}^2}{\mathrm{d}r^2}-\frac{1}{r}\frac{\mathrm{d}}{\mathrm{d}r}+\frac{m^2}{r^2}-\widetilde V/\gamma_0\right)&\text{ as }m\le 0,
\end{align}
where $\gamma_0:=1$.
\smallskip

Now, let $m>0$. 
We introduce the numbers $r_m,\, r_m'\in (0,1]$
in the following way:
\begin{itemize} 
\item 
If $\Phi(1)\ge 2m$, then $r_m, r_m'$  are such numbers that $\Phi(r_m)=m/2$ and $\Phi(r_m')=2m$ holds; obviously $r_m,r_m'\in (0,1]$.
\smallskip

\item 
If $m/2<\Phi(1)< 2m$, then $r_m$ is defined as above, while $r_m':=1$.
\smallskip

\item
If $\Phi(1)\le m/2$, we set $r_m=r_m':=1$.
\end{itemize}
It is easy to see that
\begin{eqnarray}
\label{magnfield1} 
m-\Phi(r)\ge \frac{m}{2}\quad\text{as}\quad r\in(0, r_m),\\
\label{magnfield2}
\Phi(r)-m\ge m\quad\text{as}\quad r\in(r_m', 1).
\end{eqnarray}

In the following, $v$ be an arbitrary function from $C^\infty_0(0,1)$  
normalized by
\begin{gather}
\label{normalization}
\|v\|^2_{\L^2((0, 1), 2\pi r dr)}=2\pi \int_0^1 v(r) r dr =1.
\end{gather}
Note, that  \eqref{normalization} imply the following simple estimate: 
\begin{gather}
\label{r_m}
\int_{r_m}^{r_m'} 
\frac{1}{r}|v(r)|^2\,dr\le \frac{1}{2\pi r_m^2}.
\end{gather}

If $\Phi(1)\le m/2$, then   inequality (\ref{magnfield1}) holds
for all $r\in(0, 1)$. Consequently,
$$Q(h_m(B, \widetilde{V}))[v] \ge 2\pi\int_0^1  r\left(|v'(r)|^2+\frac{m^2}{4r^2} |v(r)|^2-\widetilde{V}(r)|v(r)|^2\right)\,dr,$$
whence, denoting $\gamma_{1}:=1/4$, we get
\begin{equation}
\label{negative0} 
h_m(B, \widetilde{V})\ge \gamma_{1}\left(-\frac{\mathrm{d}^2}{\mathrm{d}r^2}-\frac{1}{r}\frac{\mathrm{d}}{\mathrm{d}r}+\frac{m^2}{r^2}-\widetilde{V}/\gamma_{1}\right)\text{ as }m>0,\,\Phi(1)\le m/2.
\end{equation}

For $\Phi(1)>m/2$ (which implies, in particular, $r_m< r_m'$)  we consider separately two cases: 
\begin{gather}
\label{1case}
\text{\bf Case  I:\quad}2\pi\int_{r_m}^{r_m'} r|v(r)|^2\,dr\le\frac{1}{2}
\\
\label{2case}
\text{\bf Case II:\quad}2\pi\int_{r_m}^{r_m'} r|v(r)|^2\,dr>\frac{1}{2}.
\end{gather}
\smallskip

\noindent\textbf{Case I.} At first we note that, due to \eqref{normalization}, \eqref{1case} is equivalent to  
\begin{gather}
\label{normalization:equiv}
2\pi\int_{(0, r_m) \cup (r_m', 1)}r |v(r)|^2\,dr> \frac{1}{2}.
\end{gather}
Inequality \eqref{normalization:equiv} together with
(\ref{magnfield1}) yields
\begin{equation}\label{1} 
2\pi\int_{(0, r_m) \cup (r_m', 1)} \frac{(m-\Phi(r))^2}{r} |v(r)|^2\,dr\ge \frac{m^2\pi}{2}\int_{(0, r_m) \cup (r_m', 1)} r|v(r)|^2\,dr \ge\frac{m^2}{8}.
\end{equation}
Combining (\ref{1}) with (\ref{magnfield1})-(\ref{magnfield2}) we find
\begin{multline*} 
\frac{1}{2\pi} Q(h_m(B, \widetilde{V}))[v]\\\ge \int_0^1r\left(|v'(r)|^2-\widetilde{V}(r)|v(r)|^2\right)\,dr+\int_{(0, r_m)\cup(r_m', 1)}\frac{(m-\Phi(r))^2}{r}|v(r)|^2\,
dr\\\ge\int_0^1r\left(|v'(r)|^2-\widetilde{V}(r)|v(r)|^2\right)\,dr+\frac{1}{2}\int_{(0, r_m)\cup(r_m', 1)}\frac{m^2}{4r}|v(r)|^2\,dr+\frac{m^2}{32\pi}. \end{multline*}
The above bound together with (\ref{r_m}) implies 
\begin{multline*}
\frac{1}{2\pi} Q(h_m(B, \widetilde{V}))[v]\ge \int_0^1r\left(|v'(r)|^2-\widetilde{V}(r)|v(r)|^2\right)\,dr\\\nonumber +\frac{1}{8}\int_{(0, r_m)\cup (r_m', 1)}\frac{m^2}{r}|v(r)|^2\,dr+\frac{r_m^2}{ {16}}
\int_{r_m}^{r_m'} \frac{m^2}{r}|v(r)|^2\,dr
\\\ge  \frac{r_1^2}{{16}}\int_0^1r\left(|v'(r)|^2+\frac{m^2}{r^2}|v(r)|^2-\frac{{16}}{r_1^2}\widetilde{V}(r)|v(r)|^2\right)\,dr.
\end{multline*}
Then, denoting $\gamma_{2}:=r_1^2/16$, we arrive at
\begin{equation}
\label{negative1} \hspace{-3mm}
h_m(B, \widetilde{V})\ge \gamma_{2}\left(-\frac{\mathrm{d}^2}{\mathrm{d}r^2}-\frac{1}{r}\frac{\mathrm{d}}{\mathrm{d}r}+\frac{m^2}{r^2}-\widetilde{V}/\gamma_2\right)\text{ as }m>0,\,\Phi(1)> m/2,\,\eqref{1case}\text{ holds}.
\end{equation}

\noindent\textbf{Case II.} 
We set $\kappa:=\mathrm{min}\{r_m'-r_m, r_m/2\}$ and fix an arbitrary
$\mu\in(0,1) $ such that
\begin{gather}\label{choice}
\mu< \frac{(r_m-\kappa) \kappa}{8(r_m'-r_m+\kappa)}
\end{gather} 
(such a choice of constants will become clear later).
Again we have two possibilities:  
\begin{gather}
\label{kappa}
\textit{Case~IIa:\quad}2\pi\int_{r_m-\kappa}^{r_m} r|v(r)|^2\,dr > \mu
\\
\label{kappa1}
\textit{Case~IIb:\quad}2\pi\int_{r_m-\kappa}^{r_m} r|v(r)|^2\,dr \le \mu.
\end{gather}

\noindent\emph{Case~IIa.} 
It is easy to see that \eqref{kappa} implies
\begin{equation}\label{intermediate}
2\pi\int_{r_m-\kappa}^{r_m} \frac{|v(r)|^2}{r}\,dr\ge \frac{\mu}{r_m^2}.
\end{equation}
Repeating the similar calculations as in Case I and taking into account (\ref{magnfield1})-(\ref{magnfield2}) and (\ref{intermediate}) we obtain the following estimate:
\begin{eqnarray}\nonumber 
\frac{1}{2\pi} Q(h_m(B, \widetilde{V}))[v]\ge \int_0^1r\left(|v'(r)|^2-\widetilde{V}(r)|v(r)|^2\right)\,dr
\\
\nonumber+\int_{(0, r_m-\kappa)\cup(r_m', 1)}\frac{(m-\Phi(r))^2}{r}|v(r)|^2\,
dr+\int_{r_m-\kappa}^{r_m}\frac{(m-\Phi(r))^2}{r}|v(r)|^2\,
dr
\\
\nonumber\ge \int_0^1r\left(|v'(r)|^2-\widetilde{V}(r)|v(r)|^2\right)\,dr\\\nonumber+\int_{(0, r_m-\kappa)\cup(r_m', 1)}\frac{m^2}{4r}|v(r)|^2\,
dr+\int_{r_m-\kappa}^{r_m}\frac{m^2}{4r}|v(r)|^2\,
dr\\\nonumber
\ge \int_0^1r\left(|v'(r)|^2-\widetilde{V}(r)|v(r)|^2\right)\,dr+\int_{(0, r_m-\kappa)\cup(r_m', 1)}\frac{m^2}{4r}|v(r)|^2\,dr\\\label{eqn.kappa}+\frac{1}{2}\int_{r_m-\kappa}^{r_m}\frac{m^2}{4r}|v(r)|^2\,dr+\frac{m^2 \mu}{16\pi r_m^2}.
\end{eqnarray}
This together with (\ref{r_m}) gives
 \begin{eqnarray*}\frac{1}{2\pi} Q(h_m(B, \widetilde{V}))[v]\ge\\\nonumber\int_0^1r\left(|v'(r)|^2-\widetilde{V}(r)|v(r)|^2\right)\,dr+\int_{(0, r_m-\kappa)\cup(r_m', 1)}\frac{m^2}{4r}|v(r)|^2\,dr\\\nonumber +\frac{1}{2}\int_{r_m-\kappa}^{r_m}\frac{m^2}{4r}|v(r)|^2\,dr+\frac{\mu}{8}\int_{r_m}^{r_m'}\frac{m^2}{r}|v(r)|^2\,dr
\\\nonumber 
\ge {\mu\over 8} \int_0^1r\left(|v'(r)|^2+ \frac{m^2}{r^2}|v|^2-\frac{8\widetilde{V}(r)}{\mu}|v|^2\right)\,dr.
\end{eqnarray*}
The latter means
\begin{gather}
\label{negative2}
\begin{array}{r}
h_m(B, \widetilde{V})\ge \displaystyle\gamma_{3}\left(-\frac{\mathrm{d}^2}{\mathrm{d}r^2}-\frac{1}{r}\frac{\mathrm{d}}{\mathrm{d}r}+\frac{m^2}{r^2}-\widetilde{V}_3/\gamma_3\right)\\[1mm]
\text{as }m>0,\,\Phi(1)> m/2,\,\eqref{2case},\, \eqref{kappa}\text{ hold}.
\end{array}
\end{gather}
where
$
\gamma_{3}={\mu\over 8}.
$
\medskip

\noindent\emph{Case~IIb.} 
We need the following auxiliary lemma.

\begin{lemma}
\label{a}
Under  assumptions (\ref{assumption1})-(\ref{assumption2}) there exists a constant $\widetilde{C}=\widetilde{C}(B)$ such that the following inequality takes place
\begin{equation}
\label{lemma}
\Phi'(r)\ge \widetilde{C} \Phi^2(r).
\end{equation}
\end{lemma}

\begin{proof}
Recall, that
\begin{gather}
\label{Phi-r}
\Phi'(r)=r B(r),
\end{gather}
where $B(r)$ is given by \eqref{assumption2}. 
One has the asymptotic formulae
\begin{gather}\label{Phi1}
\Phi(r)=
\begin{cases}\displaystyle
\frac{M}{(\alpha-1)}\frac{1}{(1-r)^{\alpha-1}}(1+{o}(1)),&\alpha>1,\\
\displaystyle M \ln(1-r) (1+{o}(1)),&\alpha=1,\\
\Phi(r)=\mathcal{O}(1),&\alpha<1
\end{cases}\text{ as }r\to 1.
\end{gather}
Taking into account that $\alpha$ is assumed to be smaller than or equal to $2$,
one easily obtains from \eqref{Phi-r} and \eqref{Phi1} the estimate \eqref{lemma} for the values of  $r$ being close to 1 (more precisely, for $r\in [r',1)$ with some $r'<1$).
Finally, for $r\in [0,r']$  one can estimate  $\Phi$ as follows:
$$(\Phi(r))^2\le\frac{ \|B\|^2_{\L^\infty(0, r')}}{4} r^4\leq
 r B(r)/\widetilde C,$$
where $\widetilde C=\displaystyle \frac{4\inf_{z\in\Omega} B(z) }{ \|B\|^2_{\L^\infty(0, r')}}$ (recall, that $\inf_{z\in\Omega} B(z)>0$ ). The lemma is proven.
\end{proof}

Let us return to the proof of the theorem. 
Recall, that we investigate \textit{Case~IIa}, which means that conditions (\ref{2case}) and (\ref{kappa1}) holds. 

In view of (\ref{kappa1}) one can choose a point $z\in (r_m-\kappa, r_m)$ such that
$$|v(z)| \le \left(\frac{\mu}{2\pi (r_m-\kappa)\kappa}\right)^{1/2}.$$
This inequality together with the fundamental theorem of calculus gives 
\begin{multline*}
\frac{1}{2}<2\pi\int_z^{r_m'} r|v(r)|^2\,dr= 2\pi\int_z^{r_m'} r\left|\int_z^r v'(t)\,\mathrm{d}t +v(z)\right|^2\,dr\\\le 4\pi\int_z^{r_m'} r(r-z) \int_z^r |v'(t)|^2\,dt\,dr +4\pi |v(z)|^2 (r_m'-z)
\\\le \frac{4\pi(r_m'-r_m+\kappa)^2}{z} \int_z^{r_m'} r |v'(r)|^2\,dr +\frac{2\mu (r_m'-r_m+\kappa)}{(r_m-\kappa) \kappa}.
\end{multline*}

Hence in view of (\ref{choice})
\begin{eqnarray}\label{derivative}
\int_z^{r_m'}r |v'(r)|^2\,dr\ge \frac{z}{16\pi(r_m'-r_m+\kappa)^2}\ge  \frac{z}{64\pi(r_m'-r_m)^2}.\end{eqnarray}
Using the mean value theorem 
$ \Phi(r_m')-\Phi(r_m)= \Phi'(r_m'')(r_m'- r_m)$,
where $r_m''$ is some point in $(r_m, r_m')$,
the monotonicity of $\Phi$ (it follows from \eqref{assumption1}), 
and Lemma (\ref{a}) we obtain
\begin{eqnarray}
\label{estimate} 
r_m'-r_m= \frac{\Phi(r_m')-\Phi(r_m)}{\Phi'(r_m'')}=\frac{3m}{2 \Phi'(r_m'')} =\frac{3\Phi(r_m)}{ \Phi'(r_m'')}\le \frac{3\Phi(r_m) }{ \widetilde{C} \Phi^2(r_m')}\le\frac{3\Phi(r_m) }{ \widetilde{C} \Phi^2(r_m)} \le \frac{6}{\widetilde{C} m}.
\end{eqnarray}
Finally, due to the choice of $\kappa$, one gets
\begin{gather}
\label{z}
z\ge r_m/2\ge r_1/2.
\end{gather}
Combining (\ref{derivative})-(\ref{z}) we conclude the existence of a constant $ C''=C''(B)>0$ such that
$$\int_z^{r_m'}r |v'(r)|^2\,dr\ge C''m^2.$$
This estimate together with (\ref{magnfield1})-(\ref{magnfield2}) and (\ref{r_m}) implies
\begin{eqnarray}
\nonumber\frac{1}{2\pi}Q(h_m(B, \widetilde{V}))[v]\ge \int_{(0, z)\cup (r_m^\prime, 1)}r|v'(r)|^2\,dr-\int_0^1r \widetilde{V}(r)|v(r)|^2\,dr\\\nonumber+  \frac{1}{2} \int_z^{r_m'}r|v'(r)|^2\,dr+ {\frac{1}{2}}C^{''} m^2\\\nonumber+\int_{(0, r_m)\cup(r_m', 1)}\frac{(m-\Phi(r))^2}{r}|v(r)|^2\,
dr+\int_{r_m}^{r_m'}\frac{(m-\Phi(r))^2}{r}|v(r)|^2\,
dr\\\nonumber
\ge \int_{(0, z)\cup (r_m', 1)}r|v'(r)|^2\,dr-\int_0^1 r\widetilde{V}(r)|v(r)|^2\,dr+ \frac{1}{2} \int_z^{r_m'}r|v'(r)|^2\,dr\\\nonumber +\int_{(0, r_m)\cup(r_m', 1)}\frac{m^2}{4r}|v(r)|^2\,
dr+  {C'' \pi r_m^2}\int_{r_m}^{r_m'}\frac{m^2}{r} |v|^2\,dr\\\nonumber\ge \gamma_{4}\int_0^1r\left(|v'(r)|^2+\frac{m^2}{r^2}|v(r)|^2- \widetilde{V}(r)/\gamma_4|v(r)|^2\right)\,dr,
\end{eqnarray}
where $\gamma_{4}:=\mathrm{min} \{1/4,\, {C'' \pi r_1^2}\}$. Thus
\begin{gather}
\label{negative final}
\begin{array}{r}
h_m(B, \widetilde{V})\ge \displaystyle\gamma_{4}\left(-\frac{\mathrm{d}^2}{\mathrm{d}r^2}-\frac{1}{r}\frac{\mathrm{d}}{\mathrm{d}r}+\frac{m^2}{r^2}-\widetilde{V}/\gamma_{4}\right)\\[1mm]
\text{as }m>0,\,\Phi(1)> m/2,\,\eqref{2case},\, \eqref{kappa1}\text{ hold}.
\end{array}
\end{gather} 
\smallskip

Combining inequalities (\ref{negative}), (\ref{negative0}), (\ref{negative1}),  (\ref{negative2}) and (\ref{negative final})  we obtain the desired estimate
$$
\forall m\in\mathbb{Z}:\quad h_m(B, \widetilde{V})\ge \gamma h_m(0, \widetilde{V}/\gamma),\quad\text{where}\quad
\gamma=\mathrm{min}\{\gamma_0,\gamma_{1},\gamma_{2},\gamma_{3},\gamma_{4}\}
$$
(note, that $\gamma$ depends only on $B$).
Consequently
\begin{gather}
\label{H:H}
H_\Omega(A,\widetilde V)\geq
\gamma
H_\Omega(0,\widetilde{V}/\gamma).
\end{gather} 
Using \eqref{H:H} and taking into account that the
spectrum of $H_\Omega(0,\widetilde{V}/\gamma)$ is purely discrete,
we conclude by the min-min principle that the spectrum of $H_\Omega(A,\widetilde{V}/\gamma)$ is also purely discrete, moreover
\begin{equation}\label{allsigma}
\forall\sigma\ge 0:\quad
\mathrm{tr}\left(H_\Omega(A, \widetilde{V})\right)_-^\sigma\le \gamma^\sigma \mathrm{tr}\left(H_\Omega(0,\widetilde{V}/\gamma)\right)_-^\sigma. 
\end{equation}
Finally, applying for $\sigma> 0$
the Lieb-Thirring bound \eqref{Lieb:Thirring} 
(recall, that in the two-dimensional case \eqref{Lieb:Thirring} holds only for positive $\sigma$) we obtain from \eqref{allsigma} the estimate 
$$\mathrm{tr}\left(H_\Omega(A, \widetilde{V})\right)_-^\sigma \le \gamma^\sigma L_{\sigma,2} \int_\Omega \left({\frac{\widetilde{V}(z)}{\gamma}}\right)^{\sigma+1}\,d z = \frac{L_{\sigma,2}}{\gamma}  \int_0^1 r \widetilde{V}^{\sigma+1}(r)\,dr,$$
where $ L_{\sigma,2}$ is a constant from \eqref{Lieb:Thirring}.  Thus Theorem~\ref{th1} is proven. Similarly, Theorem~\ref{th2} follows from 
\eqref{allsigma} (with $\sigma=0$) and  the Chadan-Khuri-Martin-Wu estimate
\eqref{Chadan}.

\subsection*{Acknowledgements}
The work of D.B. is supported by the Czech Science Foundation (GACR) within the project 17-01706S.
The work of B.S. is supported by the research project "Numbers, Geometry and Physics".

\end{document}